\documentclass[a4paper,  leqno,  myheadings,  11pt,  twoside]{article}
\usepackage{mathrsfs}

\usepackage{amsmath,  amsthm,  amssymb,  amscd,  amsxtra, graphicx}
\usepackage{latexsym,  amsfonts}

\DeclareMathOperator*{\esssup}{esssup}

\newcommand{\norm}[1]{\lVert#1\rVert}
\newcommand{\lnorm}[1]{{\lVert#1\rVert}_\infty}
\def\qc{quasi\-conformal }

\def\mcn{\mathcal{N}}
\def\md{\Delta}
\def\mc{\mathbb{C}}

\def\msu{\mathscr{U}}

\def\T{Teich\-m\"ul\-ler }

\def\bmu{\emu_Z}

\def\nmu{\|\mu\|_\infty}
\def\nzmu{\|\zmu\|}

\def\nmu{\|\mu\|_\infty}

\def\wt{\widetilde}
\def\vp{\varphi}

\def\vp{\varphi}

\def\ov{\overline}

\def\de{\md}
\def\q1s{Q^1(S)}

\def\ts{T(S)}

\def\zs{Z(S)}

\def\tde{T(\de)}

\def\emu{[\mu]}

\def\emb{[\mu]_Z}

\def\zmu{[\mu]_Z}

 \makeatletter
\@addtoreset{equation}{section}

\newtheorem{theorem}{Theorem}
\newtheorem{cor}{Corollary}[section]
\newtheorem{lemma}{Lemma}[section]

\newtheorem{theo}{Theorem}

\newtheorem{prop}{Proposition}[section]
\newtheorem*{prob}{Problem}

\newtheorem{rem}{Remark}

\begin{document}

\title{\bf{Infinitesimally extremal Beltrami differentials of non-landslide type}
\author{GUOWU YAO
}}
\date{}
 \date{July 23,  2015}
\maketitle
\begin{abstract}\noindent
In this paper,  it is shown that there are  infinitely many
extremal Beltrami differentials of non-landside type and non-constant modulus in an infinitesimal
equivalence class unless  the class contains a unique extremal.

\end{abstract}
\renewcommand{\thefootnote}{}
%\footnote{$^*$Corresponding author.}

\footnote{{2010 \it{Mathematics Subject Classification.}} Primary
30C75,  30C62.} \footnote{{\it{Key words and phrases.}} \T space,
infinitesimally extremal, non-landslide extremal.}

 \footnote{The   author was
supported by the National Natural Science Foundation of China
(Grant No. 11271216).}

\begin{centering}\section{\!\!\!\!\!{. }
Introduction}\label{S:intr}\end{centering}

%$\mathscr{HFG},\mathfrak{HFG}$

 Let $S$ be a plane domain with at least two boundary points.  The \T space $\ts$ is the space of
equivalence classes of \qc maps $f$ from $S$ to a variable domain
$f(S)$. Two \qc maps $f$ from $S$ to $f(S)$ and $g$ from $S$ to
$g(S)$ are equivalent if there is a conformal map $c$ from $f(S)$
onto $g(S)$ and a homotopy through \qc maps $h_t$ mapping $S$ onto
$g(S)$ such that $h_0=c\circ f$, $h_1=g$ and $h_t(p)=c\circ
f(p)=g(p)$ for every $t\in [0,1]$ and every $p$ in the boundary of
$S$. Denote by $[f]$  the \T equivalence class of $f$; also
sometimes denote the equivalence class by $\emu$ where $\mu$ is the
Beltrami differential of $f$.

Denote by $Bel(S)$ the Banach space of Beltrami differentials
$\mu=\mu(z)d\bar z/dz$ on $S$ with finite $L^{\infty}$-norm and by
$M(S)$ the open unit ball in $Bel(S)$.

For $\mu\in M(S)$, define
\begin{equation*}
k_0(\emu)=\inf\{\|\nu\|_\infty:\,\nu\in\emu\}.
\end{equation*}
We say that $\mu$ is extremal  in $\emu$
  if $\nmu=k_0(\emu)$, and uniquely extremal if $\|\nu\|_\infty>k_0(\mu)$ for any other
$\nu\in\emu$.

The cotangent space to $\ts$ at the basepoint is the Banach space
$Q(S)$ of integrable holomorphic quadratic differentials on $S$ with
$L^1-$norm
\begin{equation*}
\|\vp\|=\iint_{S}|\vp(z)|\, dxdy<\infty.
 \end{equation*}  In what follows,  let $\q1s$
denote the unit sphere of $Q(S)$.

 Two Beltrami differentials $\mu$ and $\nu$ in $Bel(S)$ are said to
be infinitesimally equivalent if
\begin{equation*}\iint_S(\mu-\nu)\vp \, dxdy=0,  \text{
for any } \vp\in Q(S).
\end{equation*}
The tangent space $\zs$ of $\ts$ at the basepoint is defined as  the quotient space of $Bel(S)$ under the equivalence
relation.  Denote by $\emb$ the equivalence class of $\mu$ in
$\zs$.  In particular, we use $\mcn(S)$ to denote the set of Beltrami differentials in $Bel(S)$ that is equivalent to 0.

$\zs$ is a Banach space and  its standard
sup-norm satisfies
\begin{equation*}
\|\emb\|=\norm{\mu}:=\sup_{\vp\in \q1s}Re\iint_S \mu\vp \,
dxdy=\inf\{\|\nu\|_\infty:\,\nu\in\emb\}.\end{equation*}
 We say that $\mu$ is extremal  (in $\emb$) if $\nmu=\|\emb\|$,
uniquely extremal if $\|\nu\|_\infty>\nmu$ for any other $\nu\in
\emb$.

A Beltrami differential $\mu$ in $Bel(S)$ is  said to be
of landslide type if there exists a non-empty open subset $E\subset S$ such
that
\begin{equation*}\esssup_{z\in
E}|\mu(z)|<\|\mu\|_\infty;\end{equation*} otherwise, $\mu$ is said
to be of non-landslide type.

The conception of ``non-landslide"  was firstly introduced by Li in \cite{Li5} for extremal Beltrami differentials. Here, we generalize the definition for general case.
In particular, a unique extremal is
naturally of non-landslide type.

Let $\de$ be the unit disk in the complex plane. In \cite{Li5}, Z. Li
 investigated non-uniqueness of extremal Beltrami differentials of
non-landslide type in a \T equivalence class in the universal \T space $\tde$ and proved the
following theorem.
\begin{theo}\label{Th:li5}
There is a  point $\emu$ in $\tde$ such that $\emu$ contains
infinitely many extremal Beltrami differentials of non-landslide
type with non-constant modulus.
\end{theo}

%The proof of the theorem by Li is lengthy and complicated. Actually, the theorem is almost trivial because numerous examples satisfying the condition of %Theorem \ref{Th:li5} can be constructed in  much simpler a way \cite{Yao5}.

In \cite{Li5}, Li  posed the following problem.

\begin{prob}
For any given $\emu$, is there always an extremal Beltrami differential $\wt \mu$ in $\emu$ which is of non-landslide type?
\end{prob}

In \cite{Fan}, Fan answered the problem affirmatively and proved that if $\emu$ contains infinitely many extremals, then there always
exist infinitely many extremals of non-landslide type in $\emu$. The author gave a more precise formulation for the problem
 in \cite{Yao5} by
use of variability set and point shift differentials, that is,

\begin{theo}\label{Th:nonls}
Let  $\emu$ be given in
$\ts$. Then either each extremal in $\emu$ is of non-landslide type
or there are infinitely many non-landslide extremals of non-constant
modulus in $\emu$.
\end{theo}

The goal of this paper is to show a strengthened  counterpart of Theorem \ref{Th:nonls} in the infinitesimal setting.
\begin{theorem}\label{Th:infnonls}
Let  $\zmu$ be given in
$\zs$. Then  $\zmu$ contains infinitely many non-landslide extremals of non-constant
modulus  unless $\zmu$ contains a unique extremal.
\end{theorem}

Unfortunately, due to the loss of  the notion of boundary map and
variability set, the required point shift differentials (see \cite{St6}) in the proof of  Theorem \ref{Th:nonls} is no longer available. To overcome the difficulty, we develop a new technique in a self-contained way.

One must not expect that there
always exists a non-landslide extremal of constant modulus in
$\zmu$. The reason is that each extremal in $\zmu$ can be of
non-constant modulus (see \cite{BLMM, Yao3}).
Therefore, we have a direct corollary.

\begin{cor}
There is a  point $\zmu$ in $\zs$ such that $\zmu$ contains infinitely many  non-landslide
extremals  and each  non-landslide
extremal in $\zmu$ is of non-constant modulus.
\end{cor}

 We will prove some lemmas in Section \ref{S:lemma}. Theorem \ref{Th:infnonls}
will be proved in Section \ref{S:infvers}. In addition,  we  define the local dilatation and discuss the landslide set  for an infinitesimal equivalence class in the last section.

\begin{centering}\section{\!\!\!\!\!{. }Some lemmas}\label{S:lemma}\end{centering}

\begin{lemma}\label{Th:reich}
Let $\nu\in Bel(\de)$. Then for any given $\epsilon>0$, there exists some $r\in (0,1)$  and  $\mu\in Bel(\de)$ such that \\
(1) $\mu\in \mcn(\de),$ (2) $\mu(z)=\nu(z)$, $z\in \de_r$, (3) $\norm{\mu|_{U_r}}_\infty<\epsilon$,\\
where $\de_r=\{z\in \de:|z|<r\}$, $r\in (0,1)$ and $U_r=\de\backslash \de_r$.
\end{lemma}
\begin{proof}
By Theorem 1.1 in \cite{Re2}, there exists a unique function $\beta(z)$, holomorphic in $\mc\backslash \ov{\de_r}$, such that
\begin{equation*}
\mu(z)=\begin{cases}\nu(z),  \quad z\in \de_r,\\
\beta(z),\quad z\in U_r,
\end{cases}
\end{equation*}
belongs to $\mathcal{N}(\de)$; namely,
\begin{equation*}
\beta(z)=-\frac{z}{\pi (1-r^2)}\iint_{\de_r}\frac{\nu(z)}{\zeta-z}d\xi d\eta,\; z\in \mc\backslash \ov{\de_r}.
\end{equation*}
To complete the proof of this lemma, it is sufficient to show that $\norm{\beta|_{U_r}}_\infty<\epsilon$ for small $r>0$.
We need to valuate $|\beta(z)|$ for $z\in U_r$. Let $z'$ be the intersection point of the segment $\ov{oz}$ with the circle $\{|\zeta|=r\}$ and $B_r=\{|\zeta-z'|<\frac{r}{2}\}$. Then  for $z\in U_r$,
\begin{align*}
&|\beta(z)|=\frac{|z|}{\pi (1-r^2)}\left|\iint_{\de_r}\frac{\nu(z)}{\zeta-z}d\xi d\eta\right|\leq\frac{\norm{\nu}_\infty}{\pi (1-r^2)}\iint_{\de_r}\frac{1}{|\zeta-z|}d\xi d\eta\\
&\leq M\iint_{\de_r}\frac{1}{|\zeta-z'|}d\xi d\eta\leq M\left(\iint_{B_r}\frac{1}{|\zeta-z'|}d\xi d\eta+\iint_{\de_r\backslash B_r}\frac{1}{|\zeta-z'|}d\xi d\eta\right)\\
&\leq M\left(\pi r+\frac{2}{r}\iint_{\de_r}d\xi d\eta\right)=3M\pi r=\frac{3\pi\norm{\nu}_\infty r}{\pi (1-r^2)},
\end{align*}
where $M=\frac{\norm{\nu}_\infty}{\pi (1-r^2)}$. This lemma follows readily.

\end{proof}

Generally, for a given  Beltrami differential $\mu\in Bel(S)$ and a
point $p\in S$, define
\begin{equation*}\label{Eq:sigrho}
\sigma(\mu,p)=\inf\{\esssup_{z\in U}|\mu(z)|:\, U \text{ is an open
 neighborhood
 in } S \text{ containing }p \}
\end{equation*}
to be the local dilatation of $\mu$ at $p$. If
$\sigma(\mu,p)=\nmu$, we call $p$ a non-landslide point of
$\mu$, otherwise, $p$ a landslide point of $\mu$.
The collection of all landslide points of $\mu$ is called the landslide set of $\mu$, %(or the corresponding \qc mapping $f$),
denoted by  $LS(\mu)$. We call $NSL(\mu)=S\backslash LS(\mu)$ the non-landslide set of $\mu$. It is possible that
$LS(\mu)=S$.   It is obvious that $LS(\mu)$ is an open subset of $S$. In particular, $\mu$ is non-landslide if and only if $LS(\mu)=\emptyset$.

\begin{rem} The local boundary dilatation of $\mu$ is defined for the boundary points (see Chapter 17 in \cite{GL}). Here, we generalize the notion for the inner points. Generally, $\sigma(\mu,p)$ and $|\mu(p)|$ are two different quantities for $p\in S$ and $\sigma(\mu,p)\geq|\mu(p)|$ for almost all $p\in S$ by Lebesgue's Theorem.

\end{rem}

Throughout the paper, we denote by $\de(\zeta,r)$ the round disk $\{z:\;|z-\zeta|<r\}$ ($r>0$).

\begin{lemma}\label{Th:land}Suppose that $\mu\in Bel({S})$ and $LS(\mu)\neq \emptyset$. Let $\alpha\in Bel(S)$ with $\norm{\alpha}_\infty\leq \norm{\mu}_\infty$.  Then, for any $\zeta \in LS(\mu)$, there exists a disk $\de(\zeta,r)$ in $LS(\mu)$ and  $\nu\in \bmu$ such that $\nu=\alpha$ on $\de(\zeta,r)$, $\norm{\nu}_\infty=\norm{\mu}_\infty$ and $\nu(z)=\mu(z)$ on $NLS(\mu)$. In particular, $\nu$ vanishes on $\de(\zeta,r)$ when $\alpha=0$.
\end{lemma}

\begin{proof}
Choose sufficiently small $\rho>0$ such that the disk $D=\{|z-\zeta|<\rho\}$ is contained in $LS(\mu)$. Restrict $\mu$ on $D$. Then $\norm{\mu|_D}_\infty<\norm{\mu}_\infty$. Let $\epsilon=\frac{\norm{\mu}_\infty-\norm{\mu|_D}_\infty}{2}$.
Applying  Lemma \ref{Th:reich} to $D$, we can find  some small $r\in (0,\rho)$ and $\chi\in\mcn(D)$ such that $\chi=\alpha-\mu$ on $\de(\zeta,r)$ and $\norm{\chi}_\infty<\epsilon$ on $D\backslash \de(\zeta,r)$. Put
\begin{equation*}
\nu(z)=\begin{cases}\chi(z)+\mu(z),  \quad &z\in D,\\
\mu(z), & z\in {S}\backslash D.
\end{cases}
\end{equation*}
Then $\nu\in\bmu$ and $\nu=\alpha$ on $\de(\zeta,r)$. It is clear that  $\norm{\nu}_\infty=\norm{\mu}_\infty$ and $\nu(z)=\mu(z)$ on $NLS(\mu)$.
\end{proof}

\begin{lemma}\label{Th:lempseudo}Suppose  $\mu\in Bel(S)$ with $\lnorm{\mu}=k$. Let  $\msu=\{\alpha\in \bmu:\; \lnorm{\alpha}=k\}$.
Then there exists a Beltrami differential $\nu\in \msu$ such that   $\nu(z)=0$  a.e. on  $LS(\nu)$.
\end{lemma}

\begin{proof} By Lemma   \ref{Th:land}, we can choose $\chi\in \msu$ such that $NLS(\chi)\neq\emptyset$.
   If in addition $LS(\chi)=\emptyset$, then  $\nu=\chi$ is the desired Beltrami differential. Otherwise, $LS(\chi)$ is a non-empty open subset in $S$. By Lemma \ref{Th:land}, there exists a Beltrami differential $\eta\in\msu$  such that \\
(1) $\eta(z)=\chi(z)$  on $NLS(\chi)$,\\
(2)  $\eta(z)=0$ on some small disk $\de(z'_0,r'_0)\subset LS(\chi)$.

%If  $\eta(z)=0$  a.e. on  $LS(\eta)$, then let $\nu=\eta$.

  Let $\Lambda_0$ denote the collection of $\alpha\in \msu$ with the following conditions:\\
(a) $\alpha(z)=\eta(z)$ on $NLS(\eta)$,\\
(b) there exists some small disk $\de(\zeta,r)\subset S$  such that $\alpha(z)=0$ on $\de(\zeta,r)$.

  It is obvious that $\eta\in \Lambda_0$. If $\alpha\in \Lambda_0$, let
\[\rho_0(\alpha)=\sup\{r:\;\alpha(z)=0 \text{ on some }\de(\zeta,r)\subset S\}.\]
Put
\[\rho_0=\sup\{\rho_0(\alpha):\; \alpha\in \Lambda_0\}.\]

   We proceed with the construction of  a sequence of Beltrami differentials $\{\chi_n\}$ in $\msu$.

  $n=0)$ Choose $\chi_0\in \Lambda_0$ such that $\rho_0(\chi_0)\geq r_0:=\frac{\rho_0}{2}$ and $\chi_0(z)=0$ on some $\de(z_0,r_0)\subset S$.  If $\chi_0(z)=0$ a.e. on $LS(\chi_0)$, then   $\nu=\chi_0$ is the desired Beltrami differential. Otherwise,  by the definition of local dilatation there exists $z'_1\in LS(\chi_0)$ such that  $\sigma(\chi_0,z'_1)\in (0, k)$.   Again by Lemma \ref{Th:land}, there exists a Beltrami differential $\eta_0\in\Lambda_0$  such that \\
(1) $\eta_0(z)=\chi_0(z)$ on $NLS(\chi_0)$,\\
(2) $\eta_0(z)=0$ on $\de(z_0,r_0)$,\\
(3)  $\eta_0(z)=0$ on some small disk $\de(z'_1,r'_1)\subset LS(\chi_0)\backslash \de(z_0,r_0)$.

 Let $\Lambda_1$ denote the collection of $\alpha\in \Lambda_0$ with the following conditions:\\
(a) $\alpha(z)=\eta_0(z)$ on $NLS(\eta_0)$,\\
(b) $\alpha(z)=0$ on $\de(z_0,r_0)$,\\
(c) there exists some small disk $\de(\zeta,r)\subset S\backslash \de(z_0,r_0)$  such that $\alpha(z)=0$ on $\de(\zeta,r)$.

  It is obvious that $\eta_0\in \Lambda_1$. If $\alpha\in \Lambda_1$, let
\[\rho_1(\alpha)=\sup\{r:\;\alpha(z)=0 \text{ on some }\de(\zeta,r)\subset S\backslash \de(z_0,r_0)\}.\]
Put
\[\rho_1=\sup\{\rho_1(\alpha):\; \alpha\in \Lambda_1\}.\]

 $n=1)$  Choose $\chi_1\in \Lambda_1$ such that $\rho_1(\chi_1)\geq r_1:=\frac{\rho_1}{2}$ and $\chi_1(z)=0$ on some $\de(z_1,r_1)\subset S$.  If $\chi_1(z)=0$ a.e. on $LS(\chi_1)$, then   $\nu=\chi_1$ is the desired Beltrami differential. Otherwise,  by the definition of local dilatation there exists $z'_2\in LS(\chi_1)$ such that  $\sigma(\chi_1,z'_2)\in (0, k)$.   Again by Lemma \ref{Th:land}, there exists a Beltrami differential $\eta_1\in\Lambda_1$  such that \\
(1) $\eta_1(z)=\chi_1(z)$ on $NLS(\chi_1)$,\\
(2) $\eta_1(z)=0$ on $\de(z_0,r_0)\cup\de(z_1,r_1)$,\\
(3)  $\eta_1(z)=0$ on some small disk $\de(z'_2,r'_2)\subset LS(\chi_1)\backslash (\de(z_0,r_0)\cup \de(z_1,r_1))$.

 Let $\Lambda_2$ denote the collection of $\alpha\in \Lambda_1$ with the following conditions:\\
(a) $\alpha(z)=\eta_1(z)$ on $NLS(\eta_1)$,\\
(b) $\alpha(z)=0$ on $\de(z_0,r_0)\cup\de(z_1,r_1)$,\\
(c) there exists some small disk $\de(\zeta,r)\subset S\backslash (\de(z_0,r_0)\cup \de(z_1,r_1))$  such that $\alpha(z)=0$ on $\de(\zeta,r)$.

  It is obvious that $\eta_1\in \Lambda_2$. If $\alpha\in \Lambda_2$, let
\[\rho_2(\alpha)=\sup\{r:\;\alpha(z)=0 \text{ on some }\de(\zeta,r)\subset S\backslash (\de(z_0,r_0)\cup \de(z_1,r_1))\}.\]
Put
\[\rho_2=\sup\{\rho_2(\alpha):\; \alpha\in \Lambda_2\}.\]

  $n\to n+1)$ If we can choose a   $\chi_n\in \Lambda_n$ such that $\chi_n(z)=0$ a.e. on $LS(\chi_n)$, then let $\nu=\chi_n$.  Otherwise, as the previous, we can find an $\eta_n\in\Lambda_n$  such that \\
(1) $\eta_n(z)=\chi_n(z)$ on $NLS(\chi_n)$,\\
(2) $\eta(z)=0$ on $\bigcup_{k=0}^{n-1}\de(z_k,r_k))$,\\
(3)  $\eta_n(z)=0$ on some small disk $\de(z'_{n+1},r'_{n+1})\subset LS(\chi_n)\backslash (\bigcup_{k=0}^{n-1}\de(z_k,r_k))$.

  Proceeding as above, we find five sequences,  $\{\chi_n\in \Lambda_n\}$, $\{\eta_n\in \Lambda_n\}$, $\{\Lambda_n\subset \msu\} $, $\{\rho_n\},$ $\{\de(z_n,r_n)\subset{S}\}$ ($r_n=\frac{\rho_n}{2}$) as follows.

   $\chi_n\in \Lambda_{n}$ satisfies $\rho_{n}(\chi_n)\geq r_n:=\frac{\rho_{n}}{2}$ and
  \[\chi_n(z)=0 \text{ on some } \de(z_n,r_n)\subset S\backslash (\bigcup_{k=0}^{n-1}\de(z_k,r_k)).\]

 $\Lambda_{n+1}$ is  the collection of $\alpha\in \Lambda_n$ with the following conditions:\\
 (a) $\alpha(z)=\eta_n(z)$ on $NLS(\eta_n)$,\\
 (b) $\alpha(z)=0$ on $\bigcup_{k=0}^{n}\de(z_k,r_k))$,\\
(c) there exists some small disk $\de(\zeta,r)\subset S\backslash (\bigcup_{k=0}^n\de(z_k,r_k))$  such that $\alpha(z)=0$ on $\de(\zeta,r)$.

It is obvious that $\eta_n\in \Lambda_{n+1}$. If $\alpha\in \Lambda_{n+1}$, let
\[\rho_{n+1}(\alpha)=\sup\{r:\;\alpha(z)=0 \text{ on }\de(\zeta,r)\subset S\backslash (\bigcup_{k=0}^{n}\de(z_k,r_k))\}.\]
Put
\[\rho_{n+1}=\sup\{\rho_{n+1}(\alpha):\; \alpha\in \Lambda_{n+1}\}.\]

It is clear that
\begin{equation}\label{Eq:rhor}\lim_{n\to \infty}\rho_n=\lim_{n\to \infty}r_n=0.\end{equation}

  Let $B_n=\bigcup^n_{k=0}\de(z_k,r_k)$, $n=0,1,\cdots$. Then
  \[ B=\lim_{n\to\infty} B_n=\bigcup^\infty_{n=0}\de(z_n,r_n)\subset S.\]

By the *-weak compactness, there exists a subsequence of
$\{\chi_n\}$, still denoted by $\{\chi_n \}$, which converges to a
limit $\nu\in Bel({S})$ in the *-weak topology, that is, for
any $\phi\in L^1(S)$,
\begin{equation}\label{Eq:weak}
\lim_{n\to\infty}\iint_{S}\chi_n(z)\phi(z)dxdy=\iint_{S}\nu(z)\phi(z)dxdy.
\end{equation}
When $\phi\in Q(S)$, since $\chi_n\in \bmu$, we have
\begin{equation*}
\iint_{S}\chi_n(z)\phi(z)dxdy=\iint_{S}\mu(z)\phi(z)dxdy \text{ for all } n.
\end{equation*}
Therefore,
\[\iint_{S}\nu(z)\phi(z)dxdy=\iint_{S}\mu(z)\phi(z)dxdy\]
for all $\phi\in Q(S)$ and hence $\nu\in\bmu$. On the other hand,
since $\chi_n$ converges to $\nu$ in the *-weak topology, it follows by the standard functional analysis theory
that\begin{equation*}
\|\nu\|_\infty\leq\liminf_{n\to\infty}\|\chi_n\|_\infty=k.
\end{equation*}

  \textit{Claim.} $\nu\in \msu$ and $\nu(z)=0$ a.e. on $LS(\nu)$.

 Firstly, by the inductive construction, we see that $\chi_{n+1}(z)=\chi_{n}(z)$ on $NLS(\chi_n)$. So it holds that   $\nu(z)=\chi_n(z)$ on $NLS(\chi_n)$ for all $n$. Since $NLS(\chi_n)\supseteq NLS(\chi)$, we have $\lnorm{\nu}=k$ and hence  $\nu\in \msu$.

Secondly, since $\chi_n(z)=0$ on $B_n$, we find that $\chi_n(z)\to 0$ for all  $z\in B$ as $n\to \infty$. Therefore, by the uniqueness of *-weak limit, we must have $\nu(z)=0$ on $B$.

  Now we show that $\nu(z)=0$ a.e. on $LS(\nu)$. Suppose to the contrary. Then there exists a point $p\in LS(\nu)$ such that  $\sigma(p,\nu)\in (0,k)$.
  Applying Lemma \ref{Th:land}, we can find  a Beltrami differential $\nu'\in\msu$  such that \\
(1) $\nu'(z)=\nu(z)$ on $NLS(\nu)$,\\
(2) $\nu'(z)=0$ on $B$,\\
(3)  $\nu'(z)=0$ on some small disk $\de(p,r)\subset S\backslash B$.

It is obvious that $\nu'$ belongs to  $\bigcap_{n=0}^\infty\Lambda_n$.
However, by  (\ref{Eq:rhor}) it contradicts the choice of $\chi_n$. The claim is proved, so is the lemma.

  \end{proof}

\begin{centering}\section{\!\!\!\!\!{. }Proof of Theorem \ref{Th:infnonls}}\label{S:infvers}\end{centering}

\begin{lemma}\label{Th:nclass}Let $S=\de$. Fix $k>0$. Then (1) there exist infinitely many Beltrami differentials $\mu\in [0]_Z$ such that each $\mu$ is non-landslide and has a constant modulus $k$;
(2) there exist infinitely many Beltrami differentials $\mu\in [0]_Z$ such that each $\mu$ is non-landslide with  $\lnorm{\mu}=k$ but $\mu$ has not a constant modulus.
\end{lemma}
\begin{proof} (1) Let $\mu_n(z)=k\frac{z^n}{|z|^n}$ for integer number $n\geq 1$.  It is easy to verify that $\mu_n\in [0]_Z$. It is obvious that each $\mu_n$ is non-landslide and has a constant modulus $k$.

(2) Let $E\subset \de$ be  a compact subset with positive measure and empty interior.
Since $\de\backslash E$ is a non-empty open subset of $\de$, we can choose two sequences $\{p_n\}$ in $\de\backslash E$ and
$\{r_n\}$ in $(0,1)$ respectively such that
\begin{equation*}
\de\backslash E=A\bigcup_{n=1}^\infty B_{r_n}(p_n)\; \text{ and \;}
B_{r_i}(p_i)\bigcap B_{r_j}(p_j)=\emptyset, \; i\neq j,
\end{equation*}  where  $A$ is a set of measure 0 and $B_{r_n}(p_n):=B_n$ denotes the round disk centered at $p_n$ with
radius $r_n$. Keep $\mu(z)=0$ on $E$. We apply (1) on
$B_{n}$ and then obtain infinitely many $\eta$ which are infinitesimally equivalent to $0|_{ B_n}$ on $B_n$ with modulus
$|\eta(z)|\equiv k$ on $B_n$. Modify $0|_{B_n}$ on every $B_n$
in such an equivalent way. Then we can find infinitely many Beltrami differentials $\mu\in [0]_Z$ such that each $\mu$ is non-landslide with  $\lnorm{\mu}=k$ but $\mu$ has not a constant modulus since $\mu(z)\equiv 0$ on $E$.

\end{proof}

\textbf{\emph{Proof of Theorem \ref{Th:infnonls}}}. Assume that the extremal in  $\zmu$ is not unique. We show that  $\zmu$ contains
infinitely many extremal Beltrami differentials of non-landslide type.

\textit{Case 1.}  $\mu$ contains an extremal of landslide
type.

By Lemma \ref{Th:lempseudo}, there is an extremal Beltrami differential in $\bmu$, say $\mu$, such that $LS(\mu)\neq \emptyset$ and $\mu(z)=0$ on $LS(\mu)$.
We start from $\mu$ to produce infinitely many extremal
Beltrami differentials of non-landslide type in $\zmu$.

 Since $LS(\mu)$ is a non-empty open subset of $S$,  we can choose two sequences $\{p_n\}$ in $LS(\mu)$ and
$\{r_n\}$ in $(0,1)$ respectively such that
\begin{equation*}
LS(\mu)=A\bigcup_{n=1}^\infty B_{r_n}(p_n)\; \text{ and \;}
B_{r_i}(p_i)\bigcap B_{r_j}(p_j)=\emptyset, \; i\neq j,
\end{equation*} where  $A$ is a set of measure 0 and  $B_{r_n}(p_n):=B_n$ denotes the round disk centered at $p_n$ with
radius $r_n$. Notice that $[\mu|_{ B_n}]_Z=[0|_{B_n}]_Z$  on every $B_{n}$ where we regard $[\mu|_{ B_n}]_Z$ as a point in the space $Z(B_n)$. We apply Lemma \ref{Th:nclass} on
$B_{n}$ and then find $\eta \in [\mu|_{ B_n}]_Z$ with modulus
$|\eta(z)|\equiv \nzmu$ on $B_n$. Modifying $\mu$ on every $B_n$
in such an equivalent way,  we then get infinitely many non-landslide extremals
in $\zmu$.

To obtain infinitely many non-landslide extremals of  non-constant modulus in $\zmu$,
we only need to modify $\mu|_B$ into $\eta|_B$ on  some disk $B$ in $LS(\mu)$, say $B=B_1$ such that $\eta|_B$ is of non-landslide type and  non-constant modulus with respect to the quantity $\nzmu$.
This can be done by applying Lemma   \ref{Th:nclass} (2).

\textit{Case 2.} Each extremal in
$\zmu$ is of non-landslide type (see \textbf{Remark} \ref{Th:rem2}).

If all extremals in $\zmu$ are  of  non-constant modulus, the proof is a fortiori.   Assume,  that $\zmu$ contains at least an extremal of constant modulus, say $\nu$.  Generally, it is known \cite{Re5,Yao1} that $\zmu$ contains at least an extremal of non-constant modulus unless $\zmu$ contains a unique extremal of constant modulus.  Say, $\mu$ is an extremal of non-constant modulus. Consider the Beltrami differentials curve $\eta_t=t\mu+(1-t)\nu$, $t\in [0,1]$. Then,  $\eta_t\in \zmu$ and $|\eta_t(z)|\leq t|\mu(z)|+(1-t)|\nu(z)|\leq \nzmu$ for all $t\in [0,1]$ and almost all $z\in S$. Hence every $\eta_t$ is  extremal in $\zmu$. It is clear that $\eta_t\neq\eta_s$  if $t\neq s$ in $[0,1]$. Moreover, for every $t\in (0,1)$, $\eta_t$ is non-landslide and of non-constant modulus. The concludes the proof of the theorem.

\begin{rem}\label{Th:rem2} It is an open problem whether there exists $\zmu$ such that the extremal in $\zmu$ is not unique and each extremal in  $\zmu$ is non-landslide.
\end{rem}

\begin{centering}\section{\!\!\!\!\!{. }On landslide set of an infinitesimal class}\label{S:appl}\end{centering}

Let  $\mu\in Bel(S)$.  Define the quantity
\begin{align*}
\sigma(\emb,p)=\inf\{\sigma(\eta,p):\, \eta \text{ is an extremal
 in } \emb \}
\end{align*}
to be the local dilatation of $\emb$ at $p\in S$.  If $\sigma(\bmu,p)=\norm{\mu}$, we call $p$ a non-landslide point of
$\bmu$; otherwise $p$ a landslide point of $\bmu$. The collection of all landslide points of $\bmu$ is called the landslide set of $\bmu$, %(or the corresponding \qc mapping $f$),
denoted by  $LS(\bmu)$. $LS(\bmu)$ is clearly an open subset in $S$.

\begin{prop}\label{Th:p42}
Let  $\emb\in \zs$. Then for any $p\in S$, we have either
$\sigma(\emb,p)=0$ or $\sigma(\emb,p)=\|\emb\|$.
\end{prop}
\begin{proof}
Suppose  $\sigma(\bmu,p)<\norm{\mu}$. Then there exists some extremal
in $\emu$, say $\mu$,  and  some small disk $B_r$ in $S$ with center
$p$ and radius $r$ such that $\esssup_{z\in B_r}|\mu(z)|<\norm{\mu}$.
Restricted on $B_r$, there exists some $\eta\in [\mu|_{B_r}]$ such
that $\|\eta\|_\infty<\norm{\mu}$ and $\eta(z)\equiv0$ on
some small disk $B_\delta$  in
$B_r$  with center $p$ and radius $\delta<r$  by Lemma \ref{Th:land}. This gives
$\sigma(\bmu,p)=0$.
\end{proof}

A Beltrami differential $\mu$ is called to be locally extremal if
for any subdomain  $G\subset S$, it is extremal in its class in
$Z(G)$. The notion of \emph{locally extremal} was first introduced
by Sheretov in \cite{Shr}.

 The
following three conditions are equivalent:\\
(i) all extremals in $\bmu$ are locally extremal;\\
(ii) each extremal in $\bmu$ is non-landslide;\\
(iii)  $\sigma(\bmu,p)=\norm{\mu}$ for all $p\in S$, i.e. $LS(\bmu)=\emptyset$.

There is no essential relation among
local extremals and non-landslide extremals. In fact, a local extremal
can be of landslide type due to the construction in Theorem 1 of
\cite{Yao2}. The example of a  local extremal (not a unique extremal)
of constant modulus was first given by Reich in \cite{Re3}.

The following theorem says that every component of $LS(\bmu)$ has no holes.
\begin{theorem}\label{Th:sig1}
Let $\emb\in \zs$. Then every component of $LS(\emb)$
 is simply-connected.
\end{theorem}
\begin{proof}If $LS(\emb)=\emptyset$, the theorem is a fortiori.
Suppose $LS(\emb)\neq\emptyset$.  Assume, by contradiction, that
some component $D$ of $LS(\emb)$ has holes. Since $LS(\emb)$ is
open, we can choose a Jordan curve $C$ in $D$ such that the Jordan
domain $J$ contoured by $C$ is contained in $S$ and  there is at
least a hole denoted by $H$ located in $J$. Notice that the hole $H$
is actually a connected component of $S\backslash LS(\emb)$.

For any point $p\in C$, by Proposition \ref{Th:p42}, there exists
some extremal $\nu_p$ in $\emb$ such that $\nu_p(z)\equiv 0$ on a
small disk $B_{r_p}(p)$ in $D$ with center $p$ and radius $r_p$.
When $p$ varies along $C$, all such $B_{r_p}(p)$ cover the curve
$c$. By the finite covering theorem, we can choose finite disks
$B_{r_{p_j}}(p_j)$ ($j=1,2\cdots,n$) covering $C$. Let $\nu_{p_j}$
($j=1,2\cdots,n$) be these corresponding extremals. Put
\begin{equation*}
\eta=\frac{\nu_{p_1}+\nu_{p_2}+\cdots+\nu_{p_n}}{n}.
\end{equation*}
Then $\eta\in\emb $ and $\eta$ is an extremal. In particular, we see
that $\sigma(\eta,p)\leq \frac{n-1}{n}\|\emb\|$ for all points $p$
in $R=\bigcup_{j=1}^{n}B_{r_{p_j}}(p_j)$. Set $\wt J=J\bigcup R$.
Then $\wt J$ is also a Jordan domain in $S$. Now, one easily shows
that $\eta$ is not extremal restricted on $\wt J$ and hence
$\sigma(\eta,p)<\|\emb\|$ for all $p \in \wt J$. Whereas, for $p\in
H$, we have $\sigma(\zmu,p)=\|\emb\|$, which  is a desired
contradiction.
\end{proof}

\renewcommand\refname{\centerline{\Large{R}\normalsize{EFERENCES}}}

\noindent \textit{Guowu Yao}\\
Department of Mathematical Sciences\\
  Tsinghua University\\Beijing,  100084,  People's Republic of
  China \\
  E-mail: \texttt{gwyao@math.tsinghua.edu.cn}

\medskip
\begin{thebibliography}{99}

\bibitem{BLMM} V. Bo\v{z}in,  N. Lakic  and  V. Markovi\'{c}
 and  M. Mateljevi\'{c}, \emph{Unique extremality,} J. Anal.  Math.
 75 (1998), 299-338.


\bibitem{Fan} J. Fan, \emph{ On extremal quasiconformal mappings of non-landslide type,} Proc. Amer. Math. Soc. 139 (2011), no. 8, 2729-2733.

\bibitem{GL}F. P. Gardiner and   N. Lakic,  \emph{Quasiconformal
\T Theory,}   Mathematical Surveys and Monographs, Vol. 76, Amer. Math. Soc. Providence,  RI,  2000.



 \bibitem{Li5}Z. Li,   \emph{A note on extremal \qc mappings,   }Sci. China,  Ser. A,
 53 (2010),  63-70.


\bibitem{Re2}E. Reich, \emph{An extremum problem for analytic
functions with area norm,} Ann. Acad. Sci. Fenn. Ser. A. I. Math. 2
(1976), 429-445.

\bibitem{Re3}E. Reich, \emph{On the uniqueness question for
Hahn-Banach extensions from the space of $L^1$ analytic functions,}
Proc. Amer. Math. Soc. 88 (1983), 305-310.

\bibitem{Re5}E. Reich, \emph{Non-uniquely extremal \qc mappings, }
Libertas Mathematica, 20 (2000), 33-38.



\bibitem{Shr}V. G. Sheretov, \emph{ Locally extremal \qc mappings,}
Soviet Math. Dokl. 21 (1980), 343,345.

 \bibitem{St6}K. Strebel, \emph{Point shift differentials and
extremal \qc mappings, } Ann. Acad. Sci. Fenn. Math.   23 (1998),
475-494.




%\bibitem{Ya1}G. W. Yao, \emph{Is there always  an extremal  \Tmapping?} J. Anal. Math.  94 (2004), 363-375.

\bibitem{Yao1}G. W. Yao and Y. Qi, \emph{On the  modulus of
 extremal Beltrami coefficients},
J. Math. Kyoto Univ.  46 (2006), 235-247.


\bibitem{Yao2}G. W. Yao, \emph{Unique extremality, local extremality and extremal non-decreasable dilatations,} Bull. Austral. Math. Soc. 75 (2007), 321-329.


\bibitem{Yao3}G. W. Yao, \emph{Existence of extremal Beltrami coefficients with non-constant modulus,}
Nagoya Math. J. 199 (2010), 1-14.


\bibitem{Yao5}G. W. Yao, \emph{Extremal Beltrami differentials of non-landslide type,} Proc. Edinb. Math. Soc.  to appear.



\end{thebibliography}
\end{document}